\documentclass[final]{amsart}
\usepackage{graphicx}

\usepackage{amsmath}
\usepackage{amsfonts}
\usepackage{amssymb}
\usepackage{color}
\usepackage{srcltx}
\usepackage{fullpage}
\usepackage{hyperref}
\usepackage{mathrsfs}

 \newtheorem{thm}{Theorem}
 \newtheorem{cor}[thm]{Corollary}
 \newtheorem{lem}[thm]{Lemma}
 \newtheorem{prop}[thm]{Proposition}
 \theoremstyle{definition}
 \newtheorem{defn}[thm]{Definition}
 \theoremstyle{remark}
 
 \newtheorem{ex}[thm]{Example}

\newcommand{\eq} [1] {\begin{equation}\label{#1}\quad}
\newcommand{\en} {\end{equation}}
\newcommand{\scal}[1]{\langle#1\rangle}
\newcommand{\norm}[1]{\left\Vert#1\right\Vert}

\newcommand{\re}{\operatorname{Re}}
\newcommand{\im}{\operatorname{Im}}

\newcommand{\C}{\mathbb{C}}

\newcommand{\R}{\mathbb{R}}

\newcommand{\spoint}{seed}
\newcommand{\spoints}{seeds}

\newcommand{\associated}{base }
\newcommand{\ASSOCIATED}{BASE }

\definecolor{louis-comment-color}{rgb}{0.8, 0.58, 0.46}

\definecolor{kristin-comment-color}{rgb}{0.7, 0, 0.7}

\usepackage{soul}

\begin{document}
%
%

\title{Singularities of \ASSOCIATED polynomials and Gau--Wu numbers}

\author[Camenga]{Kristin A. Camenga}
\address{Department of Mathematics\\
Juniata College\\
Huntingdon, PA\\ USA}
\email{camenga@juniata.edu}
\author[Deaett]{Louis Deaett}
\address{Department of Mathematics and Computer Sciences\\Quinnipiac University\\USA}
\email{louis.deaett@quinnipiac.edu}
\author[Rault]{Patrick X. Rault}
\address{Department of Mathematics\\ University of Nebraska at Omaha\\
6001 Dodge Street\\
Omaha, NE 68182\\
USA}
\email{prault@unomaha.edu}
\author[Sendova]{\\ Tsvetanka Sendova}
\address{Department of Mathematics\\ Michigan State University\\
East Lansing, MI 48824\\ USA}
\email{tsendova@math.msu.edu}
\author[Spitkovsky]{Ilya M. Spitkovsky}
\address{Division of Science \\ New York University Abu Dhabi (NYUAD)\\ Abu Dhabi 129188\\
United Arab Emirates}
\email{ims2@nyu.edu, imspitkovsky@gmail.com}
\author[Yates]{Rebekah B. Johnson Yates}
\address{Department of Mathematics\\Houghton College\\1 Willard Ave.\\Houghton, NY 14744\\USA}
\email{rebekah.yates@houghton.edu}

\thanks{The work on this paper was prompted by the authors' discussions during a SQuaRE workshop in May 2013 and continued during a REUF continuation workshop in August 2017, both at the American Institute of Mathematics (AIM) and supported by the NSF.  AIM also provided financial support for an additional meeting finalizing this paper.  The third author [PXR] was partially supported by the National Science Foundation grant number DMS-148695 through the Center for Undergraduate Research in Mathematics (CURM).  The fifth author [IMS] was supported in part by Faculty Research funding from the Division of Science and Mathematics, New York University Abu Dhabi.}

\subjclass{Primary 15A60}

\keywords{Numerical range, field of values, $4\times 4$ matrices, Gau--Wu number, boundary generating curve, irreducible, singularity}


\dedicatory{}

\begin{abstract}
In 2013, Gau and Wu introduced a unitary invariant, denoted by $k(A)$, of an $n\times n$ matrix $A$, which counts the maximal number of orthonormal vectors $\textbf x_j$ such that the scalar products $\scal{A\textbf x_j,\textbf x_j}$ lie on the boundary of the numerical range $W(A)$. We refer to $k(A)$ as the Gau--Wu number of the matrix $A$. In this paper we take an algebraic geometric approach and consider the effect of the singularities of the base curve, whose dual is the boundary generating curve, to classify $k(A)$.  This continues the work of Wang and Wu \cite{WangWu13} classifying the Gau-Wu numbers for $3\times 3$ matrices. Our focus on singularities is inspired by Chien and Nakazato \cite{ChiNa12}, who classified $W(A)$ for $4\times 4$ unitarily irreducible $A$ with irreducible base curve according to singularities of that curve.  When $A$ is a unitarily irreducible $n\times n$ matrix, we give necessary conditions for $k(A) = 2$, characterize $k(A) = n$, and apply these results to the case of unitarily irreducible $4\times 4$ matrices. However, we show that knowledge of the singularities is not sufficient to determine $k(A)$ by giving examples of unitarily irreducible matrices whose base curves have the same types of singularities but different $k(A)$. In addition, we extend Chien and Nakazato's classification to consider unitarily irreducible $A$ with reducible base curve and show that we can find corresponding matrices with identical base curve but different $k(A)$. Finally, we use the recently-proved Lax Conjecture to give a new proof of a theorem of Helton and Spitkovsky \cite{HelSpit}, generalizing their result in the process.

 \end{abstract}

\maketitle

\section{Introduction}\label{s:pre}

For a matrix $A\in M_n(\C)$, the set of $n\times n$ matrices with complex entries, and the scalar product $\scal{\cdot,\cdot}$ on $\C^n$ along with the norm $\norm{\cdot}$ associated with it, the \emph{numerical range of $A$} is the set
\begin{equation*}
W(A)=\left\{\scal{A\textbf x,\textbf x}\colon \norm{\textbf x}=1, \textbf x\in \C^n\right\}.
\end{equation*}
The numerical range has been studied extensively, with much recent focus on characterizing $W(A)$ for certain classes of matrices, and on its connection with various invariants or algebraic objects associated with $A$. In 2013, Gau and Wu \cite{GauWu13}
 introduced the following unitary invariant of a matrix $A\in M_n(\mathbb C)$, which can be used to describe certain properties of the numerical range $W(A)$.

\begin{defn}\label{defn:GauWu}Let $A\in M_n(\C)$.  We define the \emph{Gau--Wu number of $A$}, denoted $k(A)$, to be the maximum size of an orthonormal set $\left\{\textbf x_1,\ldots,\textbf x_k\right\}\subset\C^n$ such that the values $\scal{A\textbf x_j,\textbf x_j}$ lie on $\partial W(A)$, the boundary of $W(A)$. \end{defn}

For a given $A\in M_n(\C)$, Gau and Wu observed that $2\leq k(A) \leq n$ \cite[Lemma 4.1]{GauWu13}.  Thus $k(A)=2$ for all $A\in M_2(\C)$. Further, Wang and Wu \cite[Proposition 2.11]{WangWu13} completely categorized the values of $k(A)$ when $n=3$ according to Kippenhahn's classification \cite{Ki, Ki08} of the shape of $W(A)$. Additionally, Lee \cite{Lee} classified $k(A)$ for unitarily reducible matrices through dimension $4$.

Inspired by Chien and Nakazato \cite{ChiNa12}, we use the following tools to study the numerical range.

Let $A\in M_n(\C)$ and let \[\re A = \left(A+A^{*}\right)/2 \quad\text{ and}\quad \im A =\left(A-A^*\right)/2i, \] where $A^*$ is the conjugate transpose of $A$.  For brevity, we will generally use the notation $H_1=\re A$ and $H_2=\im A$. Note that $H_1$ and $H_2$ are both Hermitian;  $H_1$ is called the Hermitian part of $A$ and $iH_2$ is called the skew-Hermitian part of $A$.

\begin{defn}\label{defn:assoccurve}
Let
\begin{equation*}
F_A(x:y:t) = \det\left(xH_1 + yH_2+tI_n\right),
\end{equation*}
 a homogeneous  polynomial of degree $n$, and let $\Gamma_{F_A}$ denote the curve $F_A(x:y:t)=0$ in projective space $\mathbb C\mathbb{P}^2$.  We call $F_A$ the {\em \associated polynomial} and $\Gamma_{F_A}$ the {\em \associated curve} for the matrix $A$; when the matrix is clear, we will suppress the subscripted $A$. Let $\Gamma^\wedge_{F}$ denote the dual curve of $\Gamma_F$;  we call $\Gamma^\wedge_{F}$ the {\em boundary generating curve} of $W(A)$. \end{defn}

Note that the real affine points on the boundary generating curve $\Gamma^\wedge_{F}$ have a natural embedding into $\C$. Kippenhahn showed that the numerical range $W(A)$ is the convex hull of this curve \cite{Ki, Ki08}.

Chien and Nakazato used singularities to classify the irreducible base curves of $4\times 4$ matrices \cite{ChiNa12}. Naturally, such matrices are unitarily irreducible.  In this paper, we study the relationship between $k(A)$ and singularities of the \associated curve.

In Section \ref{s:projective}, we include a primer for the projective geometry required to work with singularities of \associated curves. In Section \ref{s:lemmas}, we generalize the concept of flat portions on the boundary of the numerical range
and, for $n\times n$ unitarily irreducible matrices $A$, give necessary conditions for $k(A)=2$ and characterize $k(A) = n$ in terms of singularities.
 In Section \ref{s:sing}, we apply these results to consider the effect of singularities on $k(A)$ for any $4\times 4$ matrix $A$.
 In particular, we show that a certain type of singularity is sufficient, but not necessary, to imply $k(A)>2$.
In Section \ref{s:factor}, we consider unitarily irreducible matrices $A$ with reducible \associated polynomials, extending Chien and Nakazato's characterization in \cite{ChiNa12}.   Throughout the paper, we give examples of matrices with different singularities and values of $k(A)$.

\section{Projective geometry}\label{s:projective}

Since Chien and Nakazato's characterization of $4\times 4$ matrices is based on the singularities of their \associated curves, in this section we translate well-known results into an algebraic geometry setting.

First we recall some basic facts from projective geometry. A point in the projective plane  ${\mathbb C}{\mathbb{P}^2}$ corresponds to a line in its dual plane  ${\mathbb C}{\mathbb{P}^2}^*$ and vice versa. Furthermore, according to the duality principle, a point $\left(x_1 : y_1 : t_1\right)$ lies on a line $a_1x+b_1y+c_1t=0$ in ${\mathbb C}{\mathbb{P}^2}$ if and only if the point $\left(a_1 : b_1 : c_1\right)$ lies on the line $x_1a + y_1b + t_1c=0$ in  ${\mathbb C}{\mathbb{P}^2}^*$. Thus two points $(x_1:y_1:t_1)$ and $(x_2:y_2:t_2)$ lie on a line $a_1x+b_1y+c_1t=0$ in ${\mathbb C}{\mathbb{P}^2}$ if and only if the lines $x_1a + y_1b + t_1c=0$  and $x_2a + y_2b + t_2c=0$ intersect at $\left(a_1 : b_1 : c_1\right)$ in  ${\mathbb C}{\mathbb{P}^2}^*$.

Chien and Nakazato mention the following result but omit a proof \cite[Remark 2.2]{ChiNa12}. This result is proved by Shapiro \cite[Proposition 1]{Shap1}, but we give a different justification.

\begin{lem}\label{lem:real}
    The base polynomial $F_A$ for $A\in M_n(\C)$ has only real coefficients.
\end{lem}

\begin{proof}
    Let $(x,y,t) \in \R^3$.  Then $xH_1 +yH_2 +t I_n$ is Hermitian and so can have only real eigenvalues. Thus its determinant $F_A(x, y, t)$ can only be
    real. Hence $F_A$ is a polynomial function from $\R^3$ to $\R$.  It follows that all coefficients of this polynomial must be real.
\end{proof}

The following lemma shows that rotating a matrix corresponds to applying the same rotation to the \associated curve.

\begin{lem}\label{lem:angles}
Let $A\in M_n(\C)$ and let $F_A$ denote the \associated polynomial for the matrix $A$.  Fix an angle $\theta$ and let $A' = e^{i\theta}A$, $x'=\re\big((x+iy)e^{i\theta}\big)$, $y'=\im\big((x+iy)e^{i\theta}\big)$, and $t'=t$.  Then $F_A(x:y:t)=F_{A'}\left(x':y':t'\right)$.  That is, the real affine part of the \associated curve for the rotated matrix $A'=e^{i\theta}A$ is a rotation by $\theta$ of the real affine part of the \associated curve for $A$.  Furthermore, the line $y'=0$ intersecting the curve $F_{A'}\left(x':y':t'\right)=0$ is the image under this rotation of the line $y=-x\tan\theta$ intersecting $F_A\left(x:y:t\right)=0$.
\end{lem}
\begin{proof}
We first show that the \associated polynomials for $A$ and $A'$ are the same:

\begin{align*}
F_{A'}\left(x':y':t'\right)& =F_{\left(H_1+iH_2\right)(\cos\theta + i\sin\theta)}\left(x':y':t'\right)\\
&=F_{\left(H_1\cos\theta-H_2\sin\theta\right)+i\left(H_1\sin\theta+H_2\cos\theta\right)}(x\cos\theta-y\sin\theta:x\sin\theta+y\cos\theta:t)\\
&=\det\big((x\cos\theta-y\sin\theta)(H_1\cos\theta-H_2\sin\theta)+(x\sin\theta+y\cos\theta)(H_1\sin\theta+H_2\cos\theta)+tI\big)\\
&=\det\big(x(\cos^2\theta+\sin^2\theta)H_1+y(\sin^2\theta+\cos^2\theta)H_2+tI\big)\\
&=F_A(x:y:t).
\end{align*}

Next, note that since $y'=x\sin\theta+ y\cos\theta$, the line $y'=0$ is equivalent to $x\sin\theta + y\cos\theta =0$. Hence solving for $y$ yields the line $y=-x\tan\theta$.
\end{proof}

We take the standard definition for the order of a curve $\mathcal C$ at a point $P$, denoted $\textmd{ord}_P(\mathcal C)$, which includes that the order is 1 if and only if $\mathcal C$ is nonsingular at $P$ \cite{ChiNa12}.  Similarly, we take the standard definition of the intersection multiplicity of two curves, which includes that the order of an intersection between a curve $\mathcal C$ and a non-tangent line at a point $P$ is $\textmd{ord}_P(\mathcal C)$.  Lastly, for the next lemma we recall the following concept:  a line $\ell$ is a support line of the numerical range $W(A)$ if and only if $\ell\cap \partial W(A)$ is either a line segment or a point \cite{Ki, Ki08}.

\begin{lem}\label{lem:mult}
Let $A \in M_n(\C)$ be a unitarily irreducible matrix and let $A=H_1+iH_2$, where $H_1,H_2$ are Hermitian matrices.  Let $F(x:y:t)=\det\left(xH_1+yH_2+tI\right)$.  The multiplicity of an eigenvalue $\lambda$ of $H_1$ is $m$ if and only if the \associated curve $\Gamma_F: F(x:y:t)=0$ has order $m$ at the point $(1:0:-\lambda)$ (lying on the line $\ell: y=0$).
  Furthermore, $\lambda_1$ and $\lambda_2$ are the minimum and maximum eigenvalues of $H_1$ if and only if the duals of the points $(1:0:-\lambda_1)$ and $(1:0:-\lambda_2)$ are respectively the left and right vertical support lines of $W(A)$.
\end{lem}

\begin{proof}
First, from \cite[Theorem 17]{Ki08}, we have that $W(A)$ does not intersect the line at infinity $t=0$.  Hence $\hat \ell=(0:1:0)$ cannot be a point in $W(A)$, and so $\ell: y=0$ is not tangent to $\Gamma_F$.  Similarly, the dual of $t=0$, which is the point $(0:0:1)$, cannot lie on $\Gamma_F$. Thus $\Gamma_F$ does not contain the origin of the affine plane, so $\Gamma_F \cap \ell\cap \{(x:y:t)\colon x=0\} = \varnothing$.

Let $P$ denote a point on both the curve $\Gamma_F$ and the line $\ell$, and recall that $\operatorname{ord}_P\left(\Gamma_F\right)$ denotes the order of $\Gamma_F$ at $P$. Then we have $\operatorname{ord}_P\left(\Gamma_F\right)>1$ if and only if $P$ is a singular point of $\Gamma_F$.  Since $\ell$ is not tangent to $\Gamma_F$, we see that $\operatorname{ord}_P\left(\Gamma_F\right)$ must be equal to the intersection multiplicity of $\Gamma_F$ with $\ell$ at $P$.

By B\'ezout's Theorem, the line $\ell$ intersects $\Gamma_F$ at $n$ points, counting intersection multiplicity.  We see that $\Gamma_F\cap \ell =\left\{(x:0:t) : \det\left(xH_1+tI\right)=0 \mbox{ and } x\ne 0\right\}$. So there is a one-to-one correspondence between these intersection points $(x:0:t)$ (with some intersection multiplicity $m$) and the eigenvalues $-t/x$ (with algebraic multiplicity the same $m$) of $H_1$. Thus an eigenvalue $\lambda$ of algebraic multiplicity $m$ corresponds to a point $(1:0:-\lambda)$ of intersection multiplicity $m$ on $\Gamma_F\cap \ell$.
By the preceding paragraph, this is equivalent to being a point of $\ell$ which is a singularity of order $m$ of the curve $\Gamma_F$.

Let $\lambda_1$ and $\lambda_2$ be the minimum and maximum eigenvalues, respectively, of $H_1$.  Since $[\lambda_1,\lambda_2]=W\left(H_1\right)=\operatorname{Re}(W(A))$, we conclude that $x=\lambda_1$ and $x=\lambda_2$ are vertical parallel support lines of $W(A)$.  The duals of these lines are $\left(1:0:-\lambda_1\right)$ and $\left(1:0:-\lambda_2\right)$ respectively.  The converse holds similarly.
\end{proof}

\section{Maximal and minimal $k(A)$ for $n\times n$ matrices}\label{s:lemmas}
In this section we will consider a generalization of flat portions on the boundary of the numerical range, which we will call \spoints{}, and how these affect the Gau--Wu number of an $n\times n$ matrix $A$.  In particular, we will use the notion of \spoints{} to give necessary conditions for when $k(A)=2$ and a characterization of when $k(A) = n$.  We begin by relating the eigenvalues of $A$ to the Gau--Wu number $k(A)$.

\begin{lem}\label{lem:k=2a}   Let $A\in M_n(\C)$ and assume that $A$ is not a $2\times 2$ normal matrix. If $k(A)=2$, then for all angles $\phi$, the maximum eigenvalue of $\re(e^{-i\phi}A)$ has multiplicity $1$.
\end{lem}

\begin{proof}
 Assume $\phi$ is an angle for which the maximum eigenvalue $\lambda_{\operatorname{max}}$ of $\re(e^{-i\phi}A)$ has a multiplicity of at least 2. 	Since $k(A)$ is invariant under rotation, we will assume without loss of generality that $\phi =0.$ Then there exist orthogonal unit vectors $\textbf v_1$ and $\textbf v_2$ with $\re A \textbf v_j=\lambda_{\operatorname{max}}\textbf v_j$, for $j=1,2$. This implies that $\langle A\textbf v_1, \textbf v_1\rangle$ and $\langle A\textbf v_2, \textbf v_2\rangle$ are points on $\partial W{\left(A\right)}\cap \ell$, where $\ell$ is the line with equation $x=\lambda_{\operatorname{max}}$.

Let $\lambda_{\operatorname{min}}$ be the smallest eigenvalue of $\re A$.
Since $\re A$ is Hermitian, if $\lambda_{\operatorname{min}}=\lambda_{\operatorname{max}}$, then $\re A$ has only a single eigenvalue and hence is a scalar matrix.  It follows that $A$ itself is normal.  By assumption, then, $n \ge 3$.  Thus there exists a unit eigenvector $\textbf v_3$ of $A$, corresponding to the eigenvalue $\lambda_{\operatorname{min}}=\lambda_{\operatorname{max}}$, that is orthogonal to both $\textbf v_1$ and $\textbf v_2$. Since $\re A\textbf v_3=\lambda_{\operatorname{max}}\textbf v_3$, we have that $\langle A\textbf v_3, \textbf v_3\rangle$ is a third point on $\partial W{\left(A\right)}\cap \ell$.  Therefore, in this case $k(A) \ge 3$.

Now suppose $\lambda_{\operatorname{min}}$ and $\lambda_{\operatorname{max}}$ are distinct.
Since $\re A$ is Hermitian, the eigenspace corresponding to $\lambda_{\operatorname{min}}$ is orthogonal to $\operatorname{span}\left\{\textbf v_1, \textbf v_2\right\}$. Let $\textbf v_3$ be a unit vector in the eigenspace of $\lambda_{\operatorname{min}}$. Then $\langle A\textbf v_3, \textbf v_3\rangle$ is a point on $\partial W{\left(A\right)}$ and again $k(A)\geq 3$.

\end{proof}

The following definition generalizes the concept of a flat portion on the boundary of the numerical range to include a degenerate case.

\begin{defn}\label{defn:seed}
    Let $A\in M_n(\C)$ with boundary generating curve $\Gamma^\wedge_F$.  A \emph{\spoint} of $A$ is a subset of $\partial W(A)$ of one of the following types:
    \begin{itemize}
        \item \textbf{Flat portion:}  the intersection between $\partial W(A)$ and a support line of $W(A)$ that contains more than one point;
        \item \textbf{Singular point:}  a point of $\Gamma^\wedge_F\cap \partial W(A)$ which is a singular point of $\Gamma^\wedge_F$.
    \end{itemize}
\end{defn}

While a seed is associated with the algebraic object of a matrix, it is describing a geometric feature of the corresponding numerical range and its underlying boundary generating curve $\Gamma_F^\wedge$.  Also, note that the dual of the support line of $W(A)$ at a seed is a singularity of the base curve $\Gamma_F$.

We will now recall the definition of a related concept from Wang and Wu \cite{WangWu13}.  For any $P\in W(A)$, they define $\mathcal H_P:=\textmd{span}\left\{\mathbf x\in \mathbb C^n\ : \ \langle A\mathbf x,\mathbf x\rangle=P||\mathbf x||^2\right\}$.  They then use this set to classify some of the same cases of $k(A)$ as we do in this paper. The following proposition makes the connection between seeds and $\mathcal{H}_P$.

\begin{prop}\label{prop:WangWuconnect} Let $A$ be a unitarily irreducible matrix of dimension at least 2 and let $P\in \partial W(A)$.  Then $\dim \mathcal H_P>1$ if and only if $P$ is on a seed of $A$.
\end{prop}

\begin{proof}
Without loss of generality we assume that $P$ is on a vertical support line of $W(A)$.  Call that line $\ell$ and note that $\ell: y=\re P:=\lambda$.

Assume first that $P\in \partial W(A)$ is on a seed of $A$.  Since $A$ is unitarily irreducible, $W(A)$ has no extreme points.  Hence by \cite[Proposition 2.2(c)]{WangWu13}, $\mathcal H_P=\mathcal H_Q$ for any $Q$ on the same seed as $P$.  We may thus assume without loss of generality that $P$ is on $\Gamma_F^\wedge$.  Then $\ell$ is the tangent line to $\Gamma_F^\wedge$ at $P$, and by the definition of seed we have that $\Gamma_F$ has order at least 2 at the point $\hat \ell=(1:0:-\lambda)$.  Then by Lemma \ref{lem:mult}, we have that $\lambda$ is an eigenvalue of the Hermitian part $H_1$ of $A$ with multiplicity at least 2.  Thus $\langle H_1 \mathbf x, \mathbf x\rangle=\lambda$ has two linearly independent solutions $\mathbf x_1$ and $\mathbf x_2$.  These are also solutions to $\re\langle A \mathbf x, \mathbf x\rangle=\lambda$.  Thus we see that $\langle A \mathbf x_1, \mathbf x_1\rangle$ and $\langle A \mathbf x_2, \mathbf x_2\rangle$ are both on $\ell$.  We invoke \cite[Proposition 2.2(c)]{WangWu13} again to see that $\mathbf x_1, \mathbf x_2\in \mathcal H_P$.

Next, assume that $\dim \mathcal H_P>1$.
Let $\mathbf x_1$ and $\mathbf x_2$ be linearly independent vectors
in $\mathcal H_P$.

By \cite[Proposition 2.2(a)]{WangWu13} we have for $j=1,2$ that $(H_1-\lambda I)\mathbf x_j=0$.  Hence $\lambda$ is an eigenvalue of $H_1$ with multiplicity at least 2.  By Lemma \ref{lem:mult}, we have that $\Gamma_F$ has order at least 2 at the point $(1:0:-\lambda)$.  This point is exactly $\hat \ell$.  Thus $\ell$ intersects $\Gamma_F^\wedge$ in at least two  points, counting multiplicity.  These two points either form a flat portion of $\partial W(A)$ or a singular point of $\Gamma_F^\wedge$ on $\partial W(A)$, which in either case is a seed.  Hence $P$ is on a seed.
\end{proof}

Using the concept of seeds, we can now give an interpretation of the necessary condition for $k(A)$ to be minimal.  By Proposition \ref{prop:WangWuconnect}, our next proposition is equivalent to Corollary 2.6 in \cite{WangWu13}. We provide an alternative proof of this result using our notation.

\begin{prop}\label{prop:boundary2}
    Let $A\in M_n(\C)$ with boundary generating curve $\Gamma^\wedge_{F}$.  If $k(A)=2$, then $A$ has no \spoints.  That is, $\partial W(A)$ has no flat portions and no singularity of $\Gamma^\wedge_{F}$ is on $\partial W(A)$.
\end{prop}

\begin{proof}
    Given a point $P\in \Gamma^\wedge_{F}$ on a \spoint, without loss of generality we may assume (via rotation) that $P$ occurs on a vertical tangent line to $\Gamma^\wedge_F$, which we denote by $\ell: x=\lambda$.  Thus $\hat P$ is a tangent line to $\Gamma_F$ at the point $\hat \ell$.  Let $m=\textmd{ord}_{\hat P}(\Gamma_F)$.  By Lemma \ref{lem:mult}, $m$ is also the algebraic multiplicity of $\lambda$ as an eigenvalue of $H_1$.  Since $H_1=\re A$ and we are assuming that $k(A)=2$, Lemma \ref{lem:k=2a} implies that $m=1$.  Thus $\hat \ell$ is a nonsingular point on $\Gamma_F$.  Therefore the intersection multiplicity of $\ell$ and $\Gamma^\wedge_F$ at $P$ is 1, so that, in particular, $P$ is a nonsingular point. This contradicts our premise that $P$ is on a \spoint.
\end{proof}

By Proposition~\ref{prop:boundary2}, any matrix $A$ with a \spoint{} has $k(A) \geq 3$.  While the converse holds for $3\times 3$ matrices \cite[Corollary 2.12]{WangWu13}, it does not hold in general, as shown in Example \ref{ex:no-sing}
below for $4\times 4$ matrices, and by Proposition \ref{prop:toes} for $n\times n$ matrices with $n\geq 5$.
However, we do know that for an $n\times n$ matrix $A$ with $n\geq 3$, $k(A)$ must be less than $n$ in the case that $A$ has no seeds \cite[Proposition 2.10]{WangWu13}.

  In addition, complex singularities do not necessarily increase $k(A)$, as the next example gives a $4\times 4$ matrix $A$ for which $\Gamma_F$ has two complex singularities, but no real singularities and $k(A)=2$.

\begin{ex}\label{ex:complexcusps}

 Let $A=\left[
\begin{array}{cccc}
 0 & 1 & 0 & 0 \\
 0 & 0 & \frac{1}{2} & 0 \\
 0 & 0 & 0 & 2 \\
 0 & 0 & 0 & 0 \\
\end{array}
\right]$.  By \cite[Theorem 3.2]{WangWu13}, this matrix is unitarily irreducible, while $k(A)=2$ by \cite[Theorem 3.10]{WangWu13}.  In Figure \ref{fig:complexcusps} we show the graph of $F(-ix,1,y)=0$ (i.e., the first coordinate is pure imaginary and we include the line at infinity), together with the graphs for $\Gamma_F$ and $\Gamma^\wedge_{F}$. We can see that the \associated curve $\Gamma_F$ has a pair of complex conjugate singularities at infinity, namely $(\pm i : 1 : 0)$, which do not give rise to any singularities on the (real) boundary generating curve.

\begin{figure}[h!]
\centering
\includegraphics[height = 1.7in]{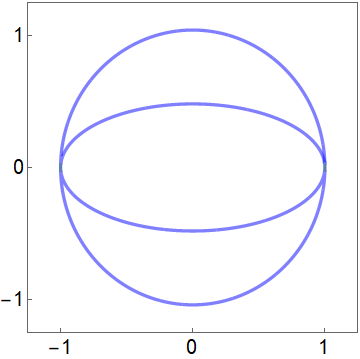}
\hspace{.5in}
\includegraphics[height = 1.7in]{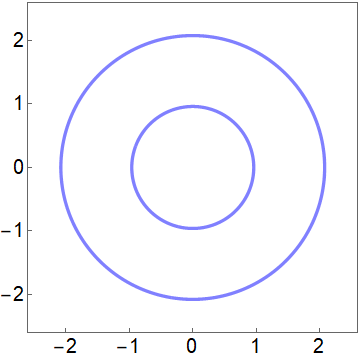}
\hspace{.5in}
\includegraphics[height = 1.7in]{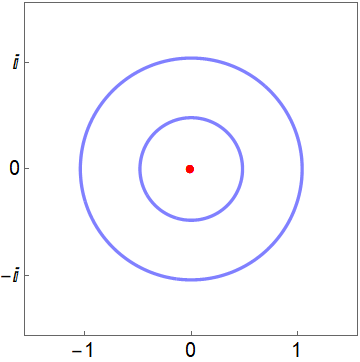}

\caption{The leftmost graph is a complex slice of $\Gamma_F$ for the matrix $A$ of Example~\ref{ex:complexcusps}. Observe that the points $(\pm i:1:0)$, which can be viewed on the left and right sides of this graph, are singularities. Despite the presence of these singularities, $k(A)=2$.  The middle and right graphs illustrate $\Gamma_F$ and $\Gamma^\wedge_{F}$ for $A$. The marked point is $0$, the unique eigenvalue of $A$.}

\label{fig:complexcusps}
\end{figure}

\end{ex}

Now we use the notion of singularities of $\Gamma_F$ to characterize those $n\times n$ matrices $A$ with maximal $k(A)$. Recall that in our notation we have $\operatorname{ord}_P\left(\Gamma_F\right)=1$ when $P$ is nonsingular.

\begin{prop}\label{lem:collinear}
    Let $n\geq 2$ and $A\in M_n(\mathbb C)$ be a unitarily irreducible matrix with \associated curve $\Gamma_F$. Then
$k(A)=n$ if and only if $\Gamma_F$ has two distinct real points $P$ and $Q$ that are collinear with the origin and satisfy $\operatorname{ord}_P\left(\Gamma_F\right)+\operatorname{ord}_Q\left(\Gamma_F\right)=n$.
\end{prop}

\begin{proof}
   {First assume that $P$ and $Q$ are two points of $\Gamma_F$ lying on a line $\ell$ through the origin, and let $m=\operatorname{ord}_P\left(\Gamma_F\right)$ and $n-m=\operatorname{ord}_Q\left(\Gamma_F\right)$.  Using Lemma \ref{lem:angles}, we assume without loss of generality that $\ell$ is the line $y=0$.  By Lemma \ref{lem:mult}, the points $P$ and $Q$ correspond to distinct eigenvalues of $H_1$, say $\lambda_1$ of multiplicity $m$ and $\lambda_2$ of multiplicity $n-m$.}

  Hence $\lambda_1$ and $\lambda_2$ are the only eigenvalues of $H_1$, and so they are in particular the maximum and minimum eigenvalues of $H_1$. Therefore, they lie on the left and right vertical support lines of $W(A)$.  Since Hermitian matrices are unitarily diagonalizable, we have a set of $n$ orthonormal eigenvectors $\left\{\textbf v_j\right\}_{j=1}^n$ for $H_1$.
  Then, for each $j\in \{1,\dots,n\}$, there exists $k\in \{1,2\}$ such that
  \begin{equation*}
  \langle A\textbf v_j, \textbf v_j \rangle = \langle (H_1 + iH_2)\textbf v_j, \textbf v_j\rangle
                                            = \langle H_1 \textbf v_j, \textbf v_j\rangle + i\langle H_2 \textbf v_j, \textbf v_j\rangle
                                            = \langle \lambda_k \textbf v_j, \textbf v_j\rangle + i\langle H_2 \textbf v_j, \textbf v_j\rangle
                                            = \lambda_k + i\langle H_2 \textbf v_j, \textbf v_j\rangle.
  \end{equation*}
  Since $H_2$ is Hermitian, $i\langle H_2 \textbf v_j, \textbf v_j\rangle$ is purely imaginary, so each $\re\langle A\textbf v_j, \textbf v_j \rangle$ for $j=1,\dots,n$ is equal to either $\lambda_1$ or $\lambda_2$. Thus each $\langle A\textbf v_j, \textbf v_j \rangle$ is on the boundary of $W(A)$, so we conclude that $k(A)=n$.

  Now assume $k(A)=n$.  By  \cite[Theorem 7]{CaRaSeS}, there exists $\phi$ such that there is a set of $n$ orthonormal vectors $\left\{\textbf v_j\right\}_{j=1}^n$ for which $\langle \re(e^{-i\phi}A)\textbf v_j,\textbf v_j\rangle \in\left\{\lambda_{\operatorname{min}}, \lambda_{\operatorname{max}}\right\}$, where $\lambda_{\operatorname{min}}$ and $\lambda_{\operatorname{max}}$ are the smallest and largest eigenvalues of $\re(e^{-i\phi}A)$, respectively. Without loss of generality, we may assume $\phi=0$.
  Then, for $j=1,\dots,n$, we have
  $\operatorname{Re}\langle A\textbf v_j,\textbf v_j\rangle = \langle H_1\textbf v_j, \textbf v_j\rangle \in \{\lambda_{\operatorname{min}},\lambda_{\operatorname{max}}\}$, and it follows by \cite[Proposition 3]{Nutshell} that $\textbf v_j$ is an eigenvector associated with $\lambda_{\operatorname{min}}$ or $\lambda_{\operatorname{max}}$.
  Thus we may conclude that the span of the eigenspaces of $\lambda_{\operatorname{min}}$ and $\lambda_{\operatorname{max}}$ is all of $\mathbb{C}^n$.  In particular, $\re A$ has at most two distinct eigenvalues. However, it cannot have only one eigenvalue as then $W(A)$ would be a vertical line segment and $A$ would be normal and thus unitarily reducible. Hence $\re A$ has two distinct eigenvalues, of multiplicities $m$ and $n-m$.  By Lemma \ref{lem:mult}, these correspond to distinct singular points of $\Gamma_F$ of order $m$ and $n-m$, respectively, on the line $y=0$.
  \end{proof}

\section{Using singularities of $\Gamma_F$ to compute $k(A)$ for $4\times 4$ matrices}\label{s:sing}

Applying Proposition \ref{lem:collinear} to Chien and Nakazato's classification, we immediately get the following result for $4\times 4$ matrices.
\begin{thm}\label{thm:maximal}
    Let $A\in M_4(\C)$.   Then $k(A)=4$ if and only if one of the following cases holds:

\begin{enumerate}
\item There is a point $P$ for which $\operatorname{ord}_P(\Gamma_F)=3$, or
\item $\Gamma_F$ has real singularities, two of which are collinear with the origin.
\end{enumerate}
\end{thm}

Observe that in the first case of Theorem \ref{thm:maximal}, the order of $\Gamma_F$ at the points $P$ and $Q$ (in the notation of Proposition \ref{lem:collinear}) will be $3$ and  $1$, so that only one of these two points is a singularity.

We can use this theorem to show that several of Chien and Nakazato's examples in \cite{ChiNa12} have $k(A)=4$. Their Example 4.1 gives an explicit matrix which falls into the first case of Theorem \ref{thm:maximal}, with a triple crossing singularity.  Examples 4.2 and 4.7 similarly provide matrices that fit the conditions of the second case of Theorem~\ref{thm:maximal}. Example 4.7 includes two parallel flat portions, whereas Example 4.2 has both types of seeds, with both on the $x$-axis.
This demonstrates the importance of using $\Gamma^\wedge_{F}$ to study $k(A)$, since one of the seeds in Example 4.2 is not a flat portion in $\partial W(A)$.

We now consider options for singularities that lead to $k(A)=3$.
In the following example, there are two singularities, but $k(A) = 3$.

\begin{ex}\label{ex:non-singk(A)=3}
Consider the matrix $$A= \left[\begin{array}{cccc}
   0 & 4i/3 & i/2 & 8i/3 \\
   0 & 0 & 8i/3 & 0 \\
   0 & 0 & 0 & 4i/3 \\
   0 & 0 & 0 & 0\\
\end{array}\right].$$
The \associated polynomial is
\[
F(x:y:t) =  t^4-\frac{649 t^2 x^2}{144}+\frac{400 x^4}{81}+\frac{8}{9} t x^2 y-\frac{649 t^2 y^2}{144}+\frac{544 x^2 y^2}{81}+\frac{8 t y^3}{9}+\frac{16 y^4}{9},
\]
which has a zero set with two singularities, both order 2 crossing singularities. We plot the graph of $F(x:y:1)=0$ in Figure \ref{fig:7ba}.  Each singularity corresponds to a flat portion of $W(A)$.  These two flat portions imply that the matrix $A$ has two \spoints{} and there are two angles $\phi$ for which $\re(e^{-i\phi}A)$ has a double maximum eigenvalue.
By Proposition \ref{prop:boundary2} or Lemma \ref{lem:k=2a}, we have $k(A)>2$.  However, the angles corresponding to these flat portions are not equal modulo $\pi$.  This means that the singularities are not collinear with the origin and $k(A)<4$ by Theorem \ref{thm:maximal}.  Therefore, $k(A) = 3$.
\begin{figure}[h!]
\centering
\includegraphics[height = 1.5in]{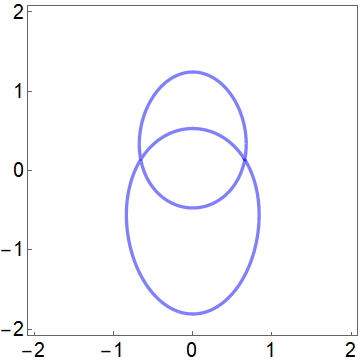}\hspace{.5in} \includegraphics[height = 1.5in]{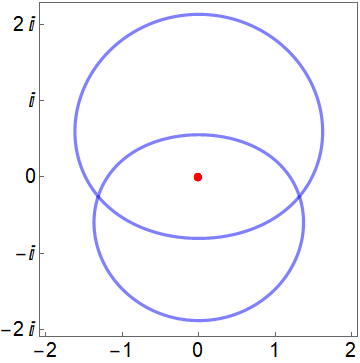}
\caption{Two real order 2 crossing singularities of $\Gamma_F$, non-collinear with the origin; boundary generating curve $\Gamma^\wedge_F$ with a quadruple eigenvalue of $A$ marked.  The matrix $A$ has $k(A) = 3$.}\label{fig:7ba}
\end{figure}
\end{ex}

The same argument can be used to show $k(A)=3$ whenever the base curve of a $4\times 4$ matrix $A$ has one or more singularities of order 2 corresponding to a \spoint{} of $A$, no pair of which are collinear with the origin.  Example 4.4 in \cite{ChiNa12} is a $4\times 4$ matrix where $\Gamma_F$ has three distinct singularities of order 2, no pair of which is collinear with the origin.  Therefore, Theorem \ref{thm:maximal} and Proposition \ref{prop:boundary2} show that $k(A)=3$.  Examples 4.3, 4.5, 4.6, and 4.9 of \cite{ChiNa12} give examples of matrices whose boundary generating curves have one singularity of order 2, which corresponds to a \spoint{} for $A$. The seed implies that $k(A) >2$ by Proposition \ref{prop:boundary2}, and since the singularity is not order 3, Theorem \ref{thm:maximal} shows that $k(A)<4$. Therefore, we can conclude that $k(A) = 3$.  The same reasoning yields the following theorem.

\begin{thm}\label{thm:k(A)=3}
    Let $A$ be a $4\times 4$ matrix.  If every real singularity of $\Gamma_F$ corresponding to a seed of $A$ is of order 2, and no two of these singularities are collinear with the origin, then $k(A)=3$.
\end{thm}

Proposition \ref{prop:boundary2} establishes that singularities of $\Gamma_F$ can sometimes be used to show that $k(A)>2$.  The following example shows that $k(A)$ can be greater than $2$ even in the absence of singularities.

\begin{ex}\label{ex:no-sing}
Consider the matrix
\[A=
\left[\renewcommand*{\arraystretch}{1.3}
\begin{array}{cccc}
 \frac{19 i}{54} & 0 & 0 & \frac{7 i}{54} \\
 0 & \frac{8}{27} & 0 & \frac{5}{54} \\
 0 & 0 & \frac{1}{2}+\frac{i}{2} & \frac{1}{18}-\frac{i}{18} \\
 \frac{7 i}{54} & \frac{5}{54} & \frac{1}{18}-\frac{i}{18} & \frac{5}{27}+\frac{10 i}{27} \\
\end{array}
\right].\]
Figure \ref{fig:a2} shows that $\Gamma_F$ has no singularities, which by Theorem \ref{thm:maximal} demonstrates that $k(A)<4$.  However, the blue points are the images of the vectors $\textbf{e}_1$, $\textbf{e}_2$, $\textbf{e}_3$, and $\textbf{e}_4$, with the first three on the boundary; this can be seen by computing the maximum eigenvalues of $\textmd{Re}(e^{i\phi}A)$ with $\phi=\pi, \pi/2, -\pi/4$.  This demonstrates that $k(A) \geq 3$.  Therefore, $k(A) = 3$ even though $\Gamma_F$ has no singularities.

\begin{figure}[h]
 \includegraphics[height=2in]{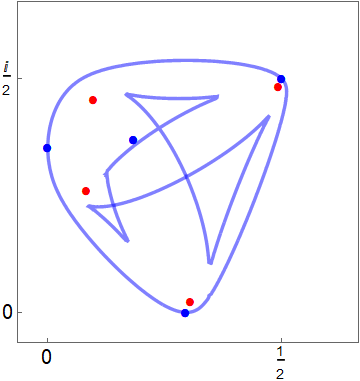}
\caption{The boundary generating curve $\Gamma_F^\wedge$ for the matrix $A$ of Example~\ref{ex:no-sing}, with $k(A) = 3$.  There are no singularities of $\Gamma_F^\wedge$ on $\partial W(A)$, but the vectors $\textbf{e}_1$, $\textbf{e}_2$ and $\textbf{e}_3$ map to the points $19i/54$, $8/27$, and $\frac{1}{2}+\frac{i}{2}$, marked by blue dots, which are on $\partial W(A)$.  The red points show the eigenvalues of $A$.}
\label{fig:a2}
\end{figure}
\end{ex}

This example shows that the classification of singularities of $\Gamma_F$ and $\Gamma^{\wedge}_F$ for a $4\times 4$ matrix $A$ is not sufficient to distinguish between matrices for which $k(A)=2$ and $k(A)=3$.

\section{Unitarily irreducible $A$ with reducible \associated polynomial}\label{s:factor}

In the preceding sections, we studied the Gau--Wu number using the algebraic tools given by the base polynomial and the boundary generating curve.  Now we will establish limitations on the power of these tools by providing an example of matrices that have the same numerical range and boundary generating curve but do not have the same Gau--Wu number.

Let $T_n(a,b,c)$ denote the $n\times n$ tridiagonal Toeplitz matrix with every entry on its diagonal given by $a$, every entry on its first superdiagonal given by $b$, and every entry on its first subdiagonal given by $c$.  We will prove the following proposition, which gives an explicit class of pairs of matrices with the same base polynomial but different Gau--Wu numbers.

\begin{prop}\label{prop:toes}
    Let $A$ denote the tridiagonal Toeplitz matrix $T_n(a,b,c)$ with $n\geq 3$ and assume that $|b|\neq |c|$, $b, c \in \mathbb{C}$.  Let $A'$ denote the matrix resulting from swapping the $a_{j,j+1}$ entry with the $a_{j+1,j}$ entry when $j$ is odd.  Then $A$ is unitarily irreducible, $A'$ is unitarily reducible, $F_A = F_{A'}$, $W(A)=W(A')$, $k(A)=\left\lceil \frac{n}{2}\right\rceil$, and $k(A')=2$.  In particular, if $n\geq 5$, $k(A) \neq k(A')$.
\end{prop}

\begin{proof}
For any tridiagonal matrix, the operation of swapping a pair of opposing off-diagonal entries preserves the numerical range; this fact is \cite[Lemma 3.1]{BS04}. We now explain how that result can be strengthened.  As noted in \cite[Section 1.6]{Prasolov}, the
influence of the off-diagonal entries on the
determinant of a tridiagonal matrix is determined
not by their individual values but only by the products of opposite off-diagonal pairs.
When $A$ is tridiagonal, so are its real and imaginary parts, and hence so is the matrix $xH_1 + yH_2+tI_n$ given in Definition \ref{defn:assoccurve}.  Therefore, the determinant of $xH_1 + yH_2+tI_n$ is invariant under the interchange of corresponding off-diagonal entries.  It follows that such an interchange, when performed on $A$, has no effect on the \associated polynomial.
Since the \associated polynomial determines the boundary generating curve, which in turn determines the numerical range, this result is stronger than the original result of \cite{BS04} mentioned above.

By shifting, we can assume without loss of generality that $a=0$, so that $A=T_n(0,b,c)$.  Note that  $A'$ is the matrix that results by swapping every entry of $A$ on the first superdiagonal lying in an odd-indexed row with its opposite subdiagonal entry.  For example, when $n=5$,
\[
A=
\begin{bmatrix}
0 & b & 0 & 0 & 0 \\
c & 0 & b & 0 & 0 \\
0 & c & 0 & b & 0 \\
0 & 0 & c & 0 & b \\
0 & 0 & 0 & c & 0
\end{bmatrix}
\quad
\text{ and }
\quad
A'=
\begin{bmatrix}
0 & c & 0 & 0 & 0 \\
b & 0 & b & 0 & 0 \\
0 & c & 0 & c & 0 \\
0 & 0 & b & 0 & b \\
0 & 0 & 0 & c & 0
\end{bmatrix}.
\]
By the above, $A$ and $A'$ have the same boundary generating polynomial, and hence $W(A)=W(A')$.  We now compare their Gau--Wu numbers.

First, note that $A$ is unitarily irreducible.
To see this, suppose some subspace $L$ is invariant under both $A$ and $A^*$.
Let $J$ be the matrix with all entries on its first superdiagonal equal to $1$, and all other entries $0$.
From the fact that $|b|\neq|c|$, it follows that $J$ and $J^*$ are linear combinations of $A$ and $A^*$. Thus $L$ is invariant under $J$ and $J^*$, and $L^\perp$ is as well.
If every vector in $L$ has its final coordinate equal to $0$, then $\textbf{e}_n \in L^\perp$, and hence $J^{n-1}\textbf{e}_n = \textbf{e}_1 \in L^\perp$.
Otherwise, there is a vector $\textbf{w}\in L$ whose final coordinate is $1$, and then $J^{n-1}\textbf{w}=\textbf{e}_1\in L$.
Thus either $L$ or $L^\perp$ contains $\textbf{e}_1$.  But then repeated application of $J^*$ shows that this subspace contains $\textbf{e}_2, \textbf{e}_3,\ldots, \textbf{e}_n$ as well.  Hence either $L$ or $L^\perp$ has dimension $n$.

On the other hand, the matrix $A'$ is unitarily reducible and is, in fact, unitarily similar to a direct sum of $\lfloor n/2 \rfloor$ blocks of size $2\times 2$, with an additional $1\times 1$ block when $n$ is odd.
This follows from the proofs of Theorem 3.3, Corollary 2.3, and then Theorem 2.1 in \cite{BS04}.
Moreover, the $1\times 1$ block, when it occurs, is necessarily zero.  Meanwhile, the $2\times 2$ blocks are given explicitly by
\begin{equation}\label{eqn:2x2_block_from_singular_value}
	A_{\sigma_j}
	=\begin{bmatrix}
		0 & \sigma_j \\
		 \beta\sigma_j & 0
	\end{bmatrix}
	=\sigma_j\begin{bmatrix}
		0 & 1 \\
		\beta & 0
	\end{bmatrix},
\end{equation}
where $\beta = c/\overline{b}$ and $\sigma_1 \ge \sigma_2 \ge \cdots \ge \sigma_{\lfloor n/2 \rfloor}$ are the nonzero singular values of the $\left\lceil \frac{n}{2}\right\rceil \times \left\lfloor \frac{n}{2}\right\rfloor$ matrix
\[
	X=
	\begin{bmatrix}
		b & 0 & ~\\
		b & b & \ddots \\
		0 & b & \ddots \\
		~ & \ddots & \ddots
	\end{bmatrix}
\]
whose $j$th column is given by $b(\textbf{e}_j +\textbf{e}_{j+1})$.
From these observations and \cite[Corollary 2.3]{BS04} it follows that the numerical range of each $A_{\sigma_j}$ is an ellipse centered at the origin.

Note that there is a unique $2\times 2$ block corresponding to the largest singular value of $X$, namely $\sigma_1$.
This holds since
 the singular values of $X$ are precisely the eigenvalues of
\[
	XX^* = |b|^2
	\begin{bmatrix}
		1 & 1 & 0 & ~ &  & ~ \\
		1 & 2 & 1 & 0 &  & ~\\
		0 & 1 & 2 & 1 & \ddots & ~ \\		
		~ & 0 &\ddots& \ddots & \ddots &0  \\
		~ & & \ddots & 1 & 2 & 1 \\		
		~& ~ &  & 0  & 1 & 1 \\		
	\end{bmatrix},
\]
and the largest eigenvalue of this matrix is simple, by the Perron-Frobenius Theorem \cite[Theorem 8.4.4]{HJ1}.
Also, the ellipse arising from the block corresponding to $\sigma_1$ contains in its interior the ellipse corresponding to each other block; this can be seen as Equation \ref{eqn:2x2_block_from_singular_value} shows that the numerical range of $W(A_{\sigma_j})$
results from dilating the numerical range of
$\displaystyle\left[\begin{smallmatrix}0&1\\\beta &0\end{smallmatrix}\right]$
by a factor of $\sigma_j$, and this numerical range is an ellipse centered at the origin.

Hence the single $2\times 2$ block corresponding to $\sigma_1$ generates an ellipse whose interior contains the ellipses arising from all other blocks, as well as the single point at the origin that gives the numerical range of the additional $1\times 1$ block when $n$ is odd.  It follows that $A'$ is unitarily equivalent to $A_{\sigma_1}\oplus B$, with $W(B)$ contained in the interior of $W(A_{\sigma_1})$.  As a result, \cite[Lemma 2.9]{WangWu13} gives that $k(A')=k(A_{\sigma_1}\oplus B) =k(A_{\sigma_1})=2$, where the last equality follows from \cite[Lemma 4.1]{GauWu13}.

At the same time, for $A$, we recall that $b$ and $c$ were chosen with $|b|\neq |c|$.  Therefore, \cite[Theorem 5]{CaRaSeS} gives $k(A)=\lceil n/2 \rceil$.  For $n\ge 5$, this gives that $k(A) \ge 3$.

To summarize, both $\Gamma^\wedge_{F_A}=\Gamma^\wedge_{F_{A'}}$ and $W(A)=W(A')$ hold for the matrices $A$ and $A'$ described above, and yet $k(A)\neq k(A')$ when $n\ge 5$.
\end{proof}

Proposition \ref{prop:toes} illustrates a limitation of the ability of the \associated polynomial and the boundary generating curve to provide information about the Gau--Wu number of a matrix.  The following is an explicit example of this situation, with a graph of the boundary generating curve.

\begin{ex}\label{ex:reducible_irreducible}
    Consider the matrix $A=T_5(0,1,2)$.  By Proposition~\ref{prop:toes}, $A$ is unitarily irreducible but has the same \associated curve as the unitarily reducible matrix
\[
A'=\begin{bmatrix}
     0 & 2 & 0 & 0 & 0 \\
     1 & 0 & 1 & 0 & 0 \\
     0 & 2 & 0 & 2 & 0 \\
     0 & 0 & 1 & 0 & 1 \\
     0 & 0 & 0 & 2 & 0 \\
   \end{bmatrix}.
\]
We show the graph of their boundary generating curve in Figure \ref{fig:5toes}.  Note that the blue point shown at the origin is both an eigenvalue and the boundary generating curve of the $1\times 1$ block in the reduction of $A'$.

\begin{figure}[h]
 \includegraphics[width=2.75in]{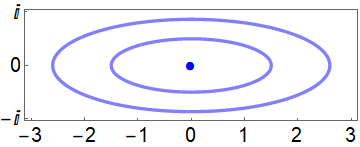}
\hspace{.5in}
\caption{The boundary generating curve for both the unitarily irreducible matrix $A$ and the unitarily reducible matrix $A'$ of Example~\ref{ex:reducible_irreducible}.  Note that $k(A)=3$ while $k(A')=2$.}
\label{fig:5toes}
\end{figure}

Proposition~\ref{prop:toes} demonstrates that $k(A)=3$, while $k(A')=2$.  As discussed in the proof of that proposition, this is because the latter is merely the Gau--Wu number of the $2\times 2$ matrix whose numerical range is the outer ellipse.

Lastly, we note that both $A$ and $A'$ have the same \associated polynomial:  $F_A(x,y,t)=F_{A'}(x,y,t)=(9 x^2 + y^2-4t^2) (27 x^2 + 3 y^2-4t^2)/16$. This polynomial is reducible, even though $A$ is unitarily irreducible.
\end{ex}

In fact, the scenario that occurs for the tridiagonal Toeplitz matrices described in Proposition~\ref{prop:toes} is an example of a more general phenomenon.  In particular, those matrices are examples for which, although $A$ is unitarily irreducible,
its \associated polynomial $F_A$ is reducible and is in fact the same as that of some unitarily reducible matrix $A'$. This happens more generally: For any unitarily irreducible matrix $A$ with reducible \associated polynomial $F_A$, there always exists
a unitarily reducible $A'$ of the same size and with the same \associated polynomial. More precisely, we have the following result, generalizing a result of Helton and Spitkovsky \cite[Theorem 4]{HelSpit}; indeed, that theorem corresponds to the special case of $F=G$ in the theorem below.

\begin{thm}\label{thm:symmetric}
    Let $A\in M_n(\C)$ and let $G$ be a (real) factor of $F_A$.  Then there exists a symmetric matrix $B$ for which $F_B=G$, where the size of $B$ is the degree of $G$.
\end{thm}

\begin{proof}
    By Lemma \ref{lem:real}, $F_A$ has only real coefficients. Chien and Nakazato observe that $F_A$ is hyperbolic with respect to $(0,0,1)$, since for all $(w_1,w_2,w_3)\in \R^3$, the polynomial $F_A(w_1,w_2,w_3-t)$ in $t$ has only real roots.  Since roots of $G$ are roots of $F_A$, we see that $G$ is also hyperbolic.  Also, as a factor of the homogenous polynomial $F_A$, $G$ is also homogenous. Let $d=\deg G$.  Since $G$ is homogenous, has real coefficients and is hyperbolic with respect to $(0,0,1)$, Lewis, Parrilo, and Ramana's proof \cite{LPR05} of the Lax Conjecture shows that there exist real symmetric (and thus Hermitian) matrices $C,D\in M_d(\R)$ for which $G=\det(xC+yD+t I_d)$.  Let $B=C+Di$, and note that $C$ is the Hermitian part of $B$ and $Di$ is the skew-Hermitian part of $B$.  Therefore, directly from Definition~\ref{defn:assoccurve}, we have $F_B=G$.  Since $C$ and $D$ are symmetric, so is $B$.
\end{proof}

According to the theorem, if we have a matrix $A$ where $F_A$ decomposes into factors, for each factor we can find a symmetric matrix whose base polynomial is that factor.  By taking the direct sum of these symmetric matrices, we obtain a block diagonal matrix $A'$ that has the same base polynomial and hence the same boundary generating curve and numerical range as $A$. This is summarized in the following corollary, in which we apply this idea to unitarily irreducible matrices.

\begin{cor}\label{cor:reducible}
Let $A$ be a unitarily irreducible matrix with reducible base polynomial $F_A$.  Then there is a block diagonal matrix $A'$, with at least two blocks, which is the same size and has the same numerical range, base polynomial, and boundary generating curve as $A$.
\end{cor}

The unitarily reducible matrix $A'$ in Corollary~\ref{cor:reducible} may have a smaller Gau-Wu number than $A$ even though the base polynomial is the same.  This was exemplified in Proposition~\ref{prop:toes} for tridiagonal Toeplitz matrices. It is also possible that the Gau--Wu number is the same for $A$ and $A'$:  For any $3\times 3$ unitarily irreducible matrix $A$ with $W(A)$ an ellipse and $k(A)=2$, we have $W(A')$ an ellipse and $k(A')=2$.  While $k(A)$ may or may not equal $k(A')$, Lee has outlined a method for computing the Gau--Wu number for  unitarily reducible matrices $A'$ when $n\leq 4$ \cite{Lee}.

\section{Summary}
In this paper we have provided a characterization of $n\times n$ matrices with $k(A)=n$ based on the singularities of the \associated curve $\Gamma_F$.  However, we have shown that even in the $4\times 4$ case, knowledge of the singularities of $\Gamma_F$ is not sufficient to completely determine $k(A)$.  Moreover, we have proved that for $n \ge 5$, no amount of knowledge of the base polynomial or \associated curve is sufficient to do so.  Any classification of $5\times 5$ matrices will necessarily have additional complexity, as then matrices with Gau--Wu numbers which are not maximal may have $k(A)=2$, $k(A)=3$, or $k(A)=4$.  Further tools will be necessary to differentiate non-maximal cases.  Lastly, we have analyzed the numerical range of unitarily irreducible matrices with reducible \associated curve.

\bibliographystyle{amsplain}
\bibliography{master}

\providecommand{\bysame}{\leavevmode\hbox to3em{\hrulefill}\thinspace}
\providecommand{\MR}{\relax\ifhmode\unskip\space\fi MR }
\providecommand{\MRhref}[2]{%
  \href{http://www.ams.org/mathscinet-getitem?mr=#1}{#2}
}
\providecommand{\href}[2]{#2}
\begin{thebibliography}{10}

\bibitem{BS04}
E.~Brown and I.~Spitkovsky, \emph{On matrices with elliptical numerical
  ranges}, Linear Multilinear Algebra \textbf{52} (2004), 177--193.

\bibitem{CaRaSeS}
K.~A. Camenga, P.~X. Rault, T.~Sendova, and I.~M. Spitkovsky, \emph{On the
  {G}au-{W}u number for some classes of matrices}, Linear Algebra Appl.
  \textbf{444} (2014), 254--262.

\bibitem{ChiNa12}
M.-T. Chien and H.~Nakazato, \emph{Singular points of the ternary polynomials
  associated with 4-by-4 matrices}, Electron. J. Linear Algebra \textbf{23}
  (2012), 755--769.

\bibitem{GauWu13}
Hwa-Long Gau and Pei~Yuan Wu, \emph{Numerical ranges and compressions of
  {$S_n$}-matrices}, Operators and Matrices \textbf{7} (2013), no.~2, 465--476.

\bibitem{HelSpit}
J.~W. Helton and I.~M. Spitkovsky, \emph{The possible shapes of numerical
  ranges}, Operators and Matrices \textbf{6} (2012), 607--611.

\bibitem{HJ1}
R.~A. Horn and C.~R. Johnson, \emph{Matrix analysis}, Cambridge University
  Press, New York, 1985.

\bibitem{Ki}
R.~Kippenhahn, \emph{{\"{U}}ber den {W}ertevorrat einer {M}atrix}, Math. Nachr.
  \textbf{6} (1951), 193--228.

\bibitem{Ki08}
\bysame, \emph{On the numerical range of a matrix}, Linear Multilinear Algebra
  \textbf{56} (2008), no.~1-2, 185--225, Translated from the German by Paul F.
  Zachlin and Michiel E. Hochstenbach.

\bibitem{Lee}
H.-Y. Lee, \emph{Diagonals and numerical ranges of direct sums of matrices},
  Linear Algebra Appl. \textbf{439} (2013), 2584--2597.

\bibitem{LPR05}
A.~S. Lewis, P.~A. Parrilo, and M.~V. Ramana, \emph{The {L}ax conjecture is
  true}, Proc. Amer. Math. Soc. \textbf{133} (2005), no.~9, 2495--2499
  (electronic).

\bibitem{Prasolov}
V.~V. Prasolov, \emph{Problems and theorems in linear algebra}, Translations of
  Mathematical Monographs, vol. 134, American Mathematical Society, Providence,
  RI, 1994, Translated from the Russian manuscript by D. A. Le\u{\i}tes.
  \MR{1277174}

\bibitem{Nutshell}
P.~J. Psarrakos and M.~J. Tsatsomeros, \emph{Numerical range: (in) a matrix
  nutshell}, Notes, National Technical University, Athens, Greece \textbf{317}
  (2004), no.~1-3, 127--141.

\bibitem{Shap1}
H.~Shapiro, \emph{On a conjecture of {K}ippenhahn about the characteristic
  polynomial of a pencil generated by two {H}ermitian matrices. {I}}, Linear
  Algebra Appl. \textbf{43} (1982), 201--221.

\bibitem{WangWu13}
K.-Z. Wang and P.~Y. Wu, \emph{Diagonals and numerical ranges of weighted shift
  matrices}, Linear Algebra Appl. \textbf{438} (2013), no.~1, 514--532.

\end{thebibliography}

\end{document}